\documentclass{article}
\usepackage{amsmath,amssymb,amsthm,color,graphicx,diagbox,tikz,colortbl,array}
\usepackage{setspace}
\usepackage{bussproofs}
\usepackage{multicol}
\usepackage{url}

\title{Very weak subintuitionistic logics}
\author{Taishi Kurahashi\footnote{Email: Kurahashi@people.kobe-u.ac.jp}
\footnote{Graduate School of System Informatics, Kobe University, 1-1 Rokkodai, Nada, Kobe 657-8501, Japan.}
and Mashu Noguchi\footnote{Email: 251x054x@gsuite.kobe-u.ac.jp}
\footnote{Graduate School of System Informatics, Kobe University, 1-1 Rokkodai, Nada, Kobe 657-8501, Japan.}}
\date{}

\theoremstyle{plain}
\newtheorem{thm}{Theorem}[section]
\newtheorem*{thm*}{Theorem}
\newtheorem{lem}[thm]{Lemma}
\newtheorem{prop}[thm]{Proposition}
\newtheorem{cor}[thm]{Corollary}
\newtheorem{fact}[thm]{Fact}
\newtheorem*{fact*}{Fact}

\newtheorem*{prob*}{Problem}
\newtheorem*{cl*}{Claim}
\newtheorem{cl}{Claim}[section]
\newtheorem*{scl*}{Subclaim}

\theoremstyle{definition}
\newtheorem{defn}[thm]{Definition}

\newtheorem{rem}[thm]{Remark}

\newcommand{\limp}{\to}
\newcommand{\liff}{\leftrightarrow}

\newcommand{\proves}{\vdash}

\newcommand{\forces}{\Vdash}
\newcommand{\nforces}{\nVdash}

\newcommand{\Logic}[1]{\mathbf{#1}}
\newcommand{\LogicVF}{\Logic{VF}}

\newcommand{\LogicWF}{\Logic{WF}}
\newcommand{\LogicF}{\Logic{F}}
 
\newcommand{\Rule}[1]{\mathrm{(#1)}}
\newcommand{\RuleRC}{\Rule{RC}}
\newcommand{\RuleRD}{\Rule{RD}}
\newcommand{\RuleRI}{\Rule{RI}}
\newcommand{\RuleMP}{\Rule{MP}}
\newcommand{\RuleAF}{\Rule{AF}}
\newcommand{\RuleRA}{\Rule{RA}}
\newcommand{\RuleRE}{\Rule{RE}}

\newcommand{\PropVar}{\mathrm{Prop}}
\newcommand{\PFml}{\mathrm{Fml}_\mathrm{P}}

\newcommand{\Subfml}{\mathrm{Sub}}
\newcommand{\lconj}{\bigwedge}
\newcommand{\ldisj}{\bigvee}

\newcommand{\MFml}{\mathrm{Fml}_\mathrm{M}}
\newcommand{\goedelTr}{\mathcal{G}}
\newcommand{\corsiTr}{\mathcal{C}}

\begin{document}

\maketitle

\begin{abstract}
    We introduce a new propositional logic, called very weak subintuitionistic logic $\LogicVF$, by adapting the relational semantics of Fitting, Marek, and Truszczy{\'n}ski for the pure logic of necessitation $\Logic{N}$ to the propositional setting.
    We prove that $\LogicVF$ and its closed negative extensions are sound and complete with respect to this semantics, and that they have the disjunction property and the finite frame property.
    We also prove that $\LogicVF$ is strictly weaker than the weak subintuitionistic logic $\LogicWF$ of Maleki and de Jongh.
    Finally, we study modal companions of $\LogicVF$ and its closed negative
    extensions via Corsi's modified G{\"o}del translation.
\end{abstract}

\section{Introduction}

In the Kripke semantics of intuitionistic propositional logic $\Logic{Int}$, frames are required to be preorders, and valuations are required to satisfy persistency.
Logics stronger than $\Logic{Int}$, that is, superintuitionistic logics, are obtained by imposing additional frame conditions (see Chagrov and Zakharyaschev \cite{CZ97}).
On the other hand, Corsi \cite{Cor87} investigated what logics arise from dropping preorder and persistency from Kripke models.
In particular, Corsi introduced the logic $\LogicF$ corresponding to the class of Kripke models in which the preorder condition and persistency are dropped.
Logics weaker than $\Logic{Int}$ are often called subintuitionistic logics.
See also Restall \cite{Res94} and Visser \cite{Vis81} for related approaches and motivations.
The study of modal companions is also important in this area.
Among other things, Corsi proved that $\Logic{K}$ is a modal companion of $\LogicF$ via Corsi's modified G{\"o}del translation.

A further development in this direction was given by Maleki and de Jongh \cite{MdJ17,dJM18}.
They introduced neighborhood semantics for propositional subintuitionistic logics, and obtained the weak subintuitionistic logic $\LogicWF$, together with a family of its extensions.
These logics are strictly weaker than $\LogicF$.

On the modal side, there is a very weak non-normal modal logic, the pure logic of necessitation $\Logic{N}$, introduced by Fitting, Marek, and Truszczy{\'n}ski \cite{FMT92}.
This logic is obtained from classical propositional logic by adding only the necessitation rule.
Fitting, Marek, and Truszczy{\'n}ski introduced a relational semantics for $\Logic{N}$ in which the accessibility relation is not a single binary relation, but a family of binary relations indexed by modal formulas.
The first author and collaborators
\cite{Kog26,KK25,Kur24,KS26,Sat25} have studied $\Logic{N}$ and its extensions in the context of provability logic, finite frame property, sequent calculus, and interpolation.

In this paper, we introduce a new logic, called very weak subintuitionistic logic $\LogicVF$, together with its extensions.
We introduce a propositional counterpart of the semantics of Fitting, Marek, and Truszczy{\'n}ski, which we call propositional FMT semantics.
Using this semantics, we show that $\LogicVF$ is strictly weaker than the weak subintuitionistic logic $\LogicWF$, which justifies the name ``very weak''.

We then study the completeness and finite frame property of extensions of
$\LogicVF$ with respect to the FMT semantics.
In this paper, we mainly consider extensions of $\LogicVF$ obtained by adding axioms of the form $\neg C$, where $\neg C$ is closed and provable in $\Logic{Int}$.
We call such formulas closed negative axioms.
For a finite set $X$ of closed negative axioms, we write $\LogicVF+X$ for the logic obtained by adding all formulas in $X$ as axioms to $\LogicVF$, and call it a closed negative extension of $\LogicVF$.
We first prove that every closed negative extension of $\LogicVF$ has the disjunction property.
Using this, we then prove that $\LogicVF+X$ is sound and complete with respect to the class of all finite FMT frames validating all elements of $X$.

Finally, we study modal companions of $\LogicVF$ and its closed negative
extensions.
We show that $\LogicVF$ corresponds to the pure logic of necessitation $\Logic{N}$ via Corsi's modified G{\"o}del translation.
More generally, if $X$ is a finite set of closed negative axioms, we associate with
$X$ a finite set $X^\ast$ of closed modal formulas and prove that both
$\Logic{N}+X^\ast$ and $\Logic{N}+X^\ast+\{\neg\Box \neg^{2k} \bot \mid k \geq 0\}$ are modal companions of $\LogicVF+X$.

The paper is organized as follows.
In Section \ref{sec:VF}, we introduce $\LogicVF$ and its closed negative extensions.
In Section \ref{sec:DP}, we prove the disjunction property of these logics.
In Section \ref{sec:FMT}, we introduce propositional FMT semantics.
We then prove soundness of $\LogicVF$ and its extensions and analyze frame conditions for closed negative axioms.
In Section \ref{sec:FFP}, we prove the finite frame property for closed negative extensions of $\LogicVF$.
In Section \ref{sec:MC}, we discuss modal companions.
In Section \ref{sect:conclude}, we conclude the present paper by mentioning several directions for future work.
Appendix is devoted to the finite frame property for closed modal extensions of $\Logic{N}$.

All of the content presented in this paper has been formalized and verified using the proof assistant Lean 4.
This formalization was carried out as part of the Formalized Formal Logic project \cite{SN26}, led by the second author, which is licensed under the Apache License 2.0.
The formalized version of the content of this paper is available at the following URL:
\begin{quote}
    \begin{center}
        \url{https://github.com/FormalizedFormalLogic/VeryWeakSubintuitionistic/releases/tag/KN26}
    \end{center}
\end{quote}

\section{The logic $\LogicVF$ and its extensions}\label{sec:VF}

In this section, we introduce the logic $\LogicVF$ and its extensions.
Let $\PropVar$ be the set of all propositional variables.
Formulas are constructed from propositional variables with the logical connectives $\land$, $\lor$, $\limp$, and $\bot$.
The connectives $\lnot$, $\top$, and $\liff$ are introduced as usual abbreviations $\lnot A \equiv A \limp \bot$, $\top \equiv \lnot \bot$, and $A \liff B \equiv (A \limp B) \land (B \limp A)$, respectively.
For a finite set of formulas $\Gamma = \{A_1,\dots,A_n\}$, we write $\lconj \Gamma$ for $A_1 \land \cdots \land A_n$ and $\ldisj \Gamma$ for $A_1 \lor \cdots \lor A_n$\footnote{As we will prove in Proposition \ref{prop:VF_aux_theorems} below, the exchange laws for $\land$ and $\lor$ hold in $\LogicVF$, as well as in the other logics under consideration. Therefore, the order of conjuncts and disjuncts need not be a concern.}.
If $\Gamma = \emptyset$, $\lconj \emptyset = \top$ and $\ldisj \emptyset = \bot$.
Let $\PFml$ denote the set of all propositional formulas.

\begin{defn}
    The logic $\LogicVF$ is defined by a Hilbert system with the following axioms and rules.
    \begin{multicols}{2}
        \begin{enumerate}
            \item $A \land B \limp A$
            \item $A \land B \limp B$
            \item $A \limp A \lor B$
            \item $B \limp A \lor B$
            \item $A \limp A$
            \item $\bot \limp A$
            \item $A \land (B \lor C) \limp (A \land B) \lor (A \land C)$
            \item \AxiomC{$A \limp B$} \AxiomC{$A \limp C$} \RightLabel{$\RuleRC$} \BinaryInfC{$A \limp B \land C$} \DisplayProof
            \item \AxiomC{$A \limp C$} \AxiomC{$B \limp C$} \RightLabel{$\RuleRD$} \BinaryInfC{$A \lor B \limp C$} \DisplayProof
            \item \AxiomC{$A \limp B$} \AxiomC{$B \limp C$} \RightLabel{$\RuleRI$} \BinaryInfC{$A \limp C$} \DisplayProof
            \item \AxiomC{$A$} \AxiomC{$A \limp B$} \RightLabel{$\RuleMP$} \BinaryInfC{$B$} \DisplayProof
            \item \AxiomC{$A$} \RightLabel{$\RuleAF$} \UnaryInfC{$B \limp A$} \DisplayProof
        \end{enumerate}
    \end{multicols}
\end{defn}

\begin{rem}\label{RA}
    Maleki and de Jongh's weak subintuitionistic logic $\LogicWF$ introduced in \cite{MdJ17} can be obtained by adding the rule $\RuleRE$ to our logic $\LogicVF$.
    \begin{prooftree}
        \AxiomC{$A \liff B$}
        \AxiomC{$C \liff D$}
        \RightLabel{$\RuleRE$}
        \BinaryInfC{$(A \limp C) \liff (B \limp D)$}
    \end{prooftree}

    Later in Theorem \ref{VF_vs_WF}, we will prove that $\LogicVF$ is strictly weaker than $\LogicWF$.
    Note that the rule \AxiomC{$A$} \AxiomC{$B$} \RightLabel{$\RuleRA$} \BinaryInfC{$A \land B$} \DisplayProof is adopted as one of the rules of $\LogicWF$ in their paper.
    We do not include it in the definition of $\LogicVF$ because it is simply admissible in $\LogicVF$ (and also in $\LogicWF$):

    \begin{prooftree}
        \AxiomC{$A$}
        \RightLabel{$\RuleAF$}
        \UnaryInfC{$\top \limp A$}
        \AxiomC{$B$}
        \RightLabel{$\RuleAF$}
        \UnaryInfC{$\top \limp B$}
        \RightLabel{$\RuleRC$}
        \BinaryInfC{$\top \limp A \land B$}
        \AxiomC{$\top$}
        \RightLabel{$\RuleMP$}
        \BinaryInfC{$A \land B$}
    \end{prooftree}
\end{rem}


\begin{prop} \label{prop:VF_aux_theorems}
    The following formulas are theorems of $\LogicVF$:
    \begin{enumerate}
        \item $(A \land B) \limp (B \land A)$,
        \item $(A \lor B) \limp (B \lor A)$,
        \item $(A \lor B) \land (A \lor C) \limp A \lor (B \land C)$.
    \end{enumerate}
\end{prop}

\begin{proof}
    1. By the axioms, we have $\LogicVF \vdash A\land B\limp B$ and $\LogicVF \vdash  A\land B\limp A$.
    Therefore, by $\RuleRC$, we obtain $\LogicVF \vdash A\land B\limp B\land A$.

    \medskip

    2. By the axioms, we have $\LogicVF \vdash A\limp B\lor A$ and $\LogicVF \vdash B\limp B\lor A$.
    Therefore, by $\RuleRD$, we obtain $\LogicVF \vdash A\lor B\limp B\lor A$.

    \medskip

    3. Let $D: \equiv A\lor(B\land C)$.
    First, by the axioms, $\LogicVF \vdash (A\lor B)\land A\limp A$ and
    $\LogicVF \vdash  A\limp D$.
    Hence, by $\RuleRI$,
    \begin{align}\label{VF1}
        \LogicVF \vdash (A\lor B)\land A\limp D.
    \end{align}

    By item 1, we have $\LogicVF \vdash (A\lor B)\land C\limp C\land(A\lor B)$.
    By the distributivity axiom, $\LogicVF \vdash C\land(A\lor B)\limp (C\land A)\lor(C\land B)$.
    Hence, by $\RuleRI$,
    \begin{align}\label{VF2}
        \LogicVF \vdash (A\lor B)\land C\limp (C\land A)\lor(C\land B).
    \end{align}

    By the axioms, $\LogicVF \vdash C\land A\limp A$ and $\LogicVF \vdash A\limp D$.
    Hence $\LogicVF \vdash C\land A\limp D$ by $\RuleRI$.
    Also, by item 1, $\LogicVF \vdash C\land B\limp B\land C$,
    and $\LogicVF \vdash B\land C\limp D$ by the axiom.
    Hence $\LogicVF \vdash C\land B\limp D$ by $\RuleRI$.
    Therefore, by $\RuleRD$, we get $\LogicVF \vdash (C\land A)\lor(C\land B)\limp D$.
    Combining this with (\ref{VF2}), we obtain
    \begin{align}\label{VF3}
        \LogicVF \vdash (A\lor B)\land C\limp D
    \end{align}
    by $\RuleRI$.

    By applying $\RuleRD$ to (\ref{VF1}) and (\ref{VF3}), $\LogicVF \vdash ((A\lor B)\land A)\lor((A\lor B)\land C)\limp D$.
    By the distributivity axiom, we have
    \[
        \LogicVF \vdash (A\lor B)\land(A\lor C) \limp ((A\lor B)\land A)\lor((A\lor B)\land C).
    \]
    Thus, we conclude $\LogicVF \vdash (A\lor B)\land(A\lor C)\limp A\lor(B\land C)$.
\end{proof}





In the present paper, we discuss not only the logic $\LogicVF$, but also
some extensions of $\LogicVF$ in particular forms.
A formula is said to be \emph{closed} if it contains no propositional variables.

\begin{defn}
    \leavevmode
    \begin{enumerate}
        \item We say $B \in \PFml$ is a \emph{closed negative axiom} if $B$ is a closed formula, $B$ is of the form $\neg C$, and $\Logic{Int}\vdash B$.

        \item For each finite set $X$ of closed negative axioms, we denote by $\LogicVF + X$ the logic obtained by extending $\LogicVF$ with axioms $B \in X$.
              We call every extension of $\LogicVF$ of this kind a \emph{closed negative extension} of $\LogicVF$.
    \end{enumerate}
\end{defn}

Notice that every closed negative extension of $\LogicVF$ is consistent because it is a sublogic of $\Logic{Int}$.
Examples of closed negative axioms are $\neg \bot$, $\neg \neg \top$, $\neg (\top\limp\neg\top)$, and so on.

\section{Disjunction property}\label{sec:DP}

Before proving the completeness and finite frame property of $\LogicVF$ and its extensions, we prove the \emph{disjunction property} (DP) of them.
We say that a logic $L$ has DP if for any formulas $A$ and $B$, if $L \proves A \lor B$, then $L \proves A$ or $L \proves B$.
One syntactic way to prove DP is using Aczel's modified version \cite{Acz68} of Kleene's slash notation \cite{Kle62}.

\begin{defn}
    For each logic $L$, we define the relation $|_L\, A$ inductively as follows:
    \begin{enumerate}
        \item $|_L\, p$ does not hold for each $p \in \PropVar$,
        \item $|_L\, \bot$ does not hold,
        \item $|_L\, A \land B : \iff |_L\, A$ and $|_L\, B$,
        \item $|_L\, A \lor B :\iff |_L\, A$ or $|_L\, B$,
        \item $|_L\, A \limp B :\iff L \proves A \limp B$ and (if $|_L\, A$, then $|_L\, B$).
    \end{enumerate}
\end{defn}

\begin{lem} \label{lem:VF_slash_lemma}
    Let $L$ be any closed negative extension of $\LogicVF$.
    For any $A \in \PFml$, $L \proves A$ if and only if $|_L \, A$.
\end{lem}

\begin{proof}
    $(\Leftarrow)$:
    This direction is proved by induction on the construction of $A$.

    \medskip

    $(\Rightarrow)$:
    We prove this direction by induction on the length of a proof of $A$ in $L$.
    It suffices to check that every axiom of $L$ is slashed and that all rules
    preserve the slash relation.
    It is easily verified that all axioms of $\LogicVF$ are slashed simply by the definition of $|_L$.

    Let $\neg C$ be a closed negative axiom of $L$.
    Suppose $|_L\,C$.
    By the already proved direction $(\Leftarrow)$, we have $L \vdash C$.
    Since $L \vdash C \limp \bot$, we obtain $L \vdash \bot$ by $\RuleMP$, contradicting the consistency of $L$.
    Thus the condition ``$|_L\,C \Rightarrow |_L\,\bot$'' is vacuously satisfied, and hence $|_L\,\neg C$.

    We show that the rules preserve the slash relation.
    For $\RuleRC$, suppose $|_L\, A \limp B$ and $|_L\, A \limp C$.
    Then $L \vdash A \limp B$ and $L \vdash A \limp C$, so by $\RuleRC$, we have $L \vdash A \limp B \land C$.
    Moreover, if $|_L\,A$, then we get $|_L\,B$ and $|_L\,C$ from $|_L\, A \limp B$ and
    $|_L\, A \limp C$.
    Hence $|_L\, B \land C$.
    We have proved $|_L\, A \limp B \land C$.
    The other cases are similar.
\end{proof}

\begin{thm}[Disjunction Property] \label{thm:VF_has_DP}
    Let $L$ be any closed negative extension of $\LogicVF$.
    Then $L$ has DP.
\end{thm}

\begin{proof}
    Suppose $L \proves A \lor B$.
    By Lemma \ref{lem:VF_slash_lemma}, $|_L\, A \lor B$.
    Then, by the definition, $|_L\, A$ or $|_L\, B$.
    By Lemma \ref{lem:VF_slash_lemma} again, we obtain $L \proves A$ or $L \proves B$.
\end{proof}

\section{FMT semantics}\label{sec:FMT}

In \cite{FMT92}, Fitting, Marek, and Truszczy{\'n}ski introduced a relational semantics for the pure logic of necessitation $\Logic{N}$, which is a very weak non-normal modal logic.
In their semantics, the accessibility relation is not given as a single binary
relation, but as a family of binary relations $\{R_A\}$ indexed by modal
formulas $A$.

In this section, by appropriately modifying their idea for the present
propositional setting, we introduce a new semantics for $\LogicVF$ and its
extensions.
We call this semantics \emph{propositional FMT semantics}, or simply \emph{FMT semantics}.

\subsection{FMT models and soundness of $\LogicVF$}

\begin{defn}[FMT frames]
    An \emph{FMT frame} $F = \langle W, \{R_A\}_{A \in \PFml}, r \rangle$ is a tuple consisting of the following.
    \begin{itemize}
        \item $W$ is a non-empty set, called the set of worlds.
        \item $R_A$ is a binary relation on $W$ indexed by a propositional formula $A$.
        \item $r \in W$ is called the root, in the sense that $r R_A x$ for all $A \in \PFml$ and $x \in W$.
    \end{itemize}
\end{defn}

\begin{defn}[FMT models]
    An \emph{FMT model} is a pair $\langle F, V \rangle$, where $F$ is an FMT frame and $V \colon W \to 2^\PropVar$ is a valuation on $F$.
    The valuation $V$ is extended to a forcing relation between worlds and
    propositional formulas, that is, a relation on $W\times \PFml$, as follows:
    \begin{itemize}
        \item $M, x \forces p :\iff p \in V(x)$ where $p \in \PropVar$,
        \item $M, x \forces \bot$ never holds,
        \item $M, x \forces A \land B : \iff M, x \forces A$ and $M, x \forces B$,
        \item $M, x \forces A \lor B : \iff M, x \forces A$ or $M, x \forces B$,
        \item $M, x \forces A \limp B : \iff$ for any $y \in W$ satisfying $x R_{A \limp B} y$, $M, y \forces A$ implies $M, y \forces B$.
    \end{itemize}

    We say that a formula $A$ is \emph{valid} in an FMT model $M$ (written $M \models A$) if $M, x \Vdash A$ for all $x \in W$.
    We say that $A$ is valid in an FMT frame $F$ (written $F \models A$) if $A$ is valid in all FMT models based on $F$.
\end{defn}

We often simply write $x \forces A$ instead of $M,x\forces A$.
We write $x \in F$ or $x \in M$ to indicate that $x$ is a world of the frame $F$ or the model $M$, respectively.
By the definition of the forcing relation, we immediately obtain the following
clauses for the abbreviations $\neg$, $\top$, and $\leftrightarrow$.

\begin{itemize}
    \item $x \forces \lnot A$ if and only if for any $y$, if $x R_{\lnot A} y$ then $y \nforces A$.
    \item $x \forces \top$ always holds.
    \item $x \forces A \liff B$ if and only if for any $y$ satisfying $x R_{A \limp B} y$, $y \forces A$ implies $y \forces B$, and for any $z$ satisfying $x R_{B \limp A} z$, $z \forces B$ implies $z \forces A$.
\end{itemize}

\begin{thm}[Soundness of $\LogicVF$ with respect to FMT semantics]
    For any $A \in \PFml$ and FMT frame $F$, if $\LogicVF \vdash A$, then
    $F \models A$.
\end{thm}

\begin{proof}
    We prove the theorem by induction on the length of a proof of $A$ in $\LogicVF$.
    Let $F=\langle W,\{R_A\}_{A\in\PFml},r\rangle$ be an arbitrary FMT frame
    and let $V$ be an arbitrary valuation on $F$.
    Let $M=\langle W,\{R_A\}_{A\in\PFml},r,V\rangle$.
    It suffices to show that every axiom of $\LogicVF$ is valid on $M$ and that
    all rules of $\LogicVF$ preserve validity on $M$.

    First, we verify that all axioms of $\LogicVF$ are valid in $M$.
    We only check the axiom $A\land B\limp A$, and the other axioms are verified
    similarly.
    Let $x\in W$ be arbitrary and let $y\in W$ be such that $xR_{A\land B\limp A}y$.
    Suppose $y\forces A\land B$.  Then $y\forces A$ by the definition of the
    forcing relation for conjunction.  Hence $x\forces A\land B\limp A$.
    Since $x$ was arbitrary, $M\models A\land B\limp A$.

    Next, we show that the rules of $\LogicVF$ preserve validity on $M$.
    We show the case $\RuleMP$.
    Suppose $M\models A$ and $M\models A\limp B$.
    Let $x\in W$ be arbitrary.
    Since $M\models A\limp B$, we have in particular $r\forces A\limp B$.
    Since $r$ is the root of $F$, we have $rR_{A\limp B}x$.
    Also, since $M\models A$, we have $x\forces A$.
    Therefore, by $r\forces A\limp B$, we obtain $x\forces B$.
    Since $x\in W$ is arbitrary, $M\models B$.

    Next, we show the case $\RuleRC$.
    Suppose $M\models A \limp B$ and $M\models A\limp C$.
    Let $x, y \in W$ be arbitrary such that $x R_{A \limp B \land C} y$ and $y \forces A$.
    Since $r$ is the root of $F$, we have $rR_{A\limp B}y$ and $r R_{A \limp C} y$.
    By the supposition, we have $r\forces A\limp B$ and $r \forces A \limp C$.
    Therefore, we obtain $y \forces B \land C$.
    Hence, $x \forces A \limp B \land C$.
    Since $x\in W$ is arbitrary, $M\models A \limp B \land C$.

    The cases $\RuleRD$, $\RuleRI$, and $\RuleAF$ are proved in the similar way.
\end{proof}

The definition of FMT frames requires the existence of a root.
As the above proof shows, this requirement is essential for the soundness of the rules of $\LogicVF$.


We now compare our $\LogicVF$ with the weak subintuitionistic logic $\LogicWF$.
Recall from Remark \ref{RA} that $\LogicWF$ is obtained by adding the rule
$\RuleRE$ to $\LogicVF$:
\begin{prooftree}
    \AxiomC{$A \liff B$}
    \AxiomC{$C \liff D$}
    \RightLabel{$\RuleRE$}
    \BinaryInfC{$(A \limp C) \liff (B \limp D)$}
\end{prooftree}
We prove that this rule is not admissible in $\LogicVF$.

\begin{thm}\label{VF_vs_WF}
    $\LogicVF$ is a proper sublogic of $\LogicWF$.
    In particular, $\LogicVF$ is not closed under the rule $\RuleRE$.
\end{thm}
\begin{proof}
    Since $\LogicWF$ is obtained from $\LogicVF$ by adding $\RuleRE$, every theorem
    of $\LogicVF$ is a theorem of $\LogicWF$.
    It is easily shown that the formulas $\top\liff\top$ and $(p\land q)\liff(q\land p)$ are provable in $\LogicVF$.
    Since $\LogicWF$ is closed under $\RuleRE$, we have
    \[
        \LogicWF \vdash
        (\top\limp(p\land q))\liff(\top\limp(q\land p)).
    \]
    We show that this formula is not provable in $\LogicVF$.

    We define an FMT model $M=\langle \{0,1,2\},\{R_A\}_{A\in\PFml},0,V\rangle$  as follows:
    \begin{itemize}
        \item $xR_{\top\limp(p\land q)}y \:\iff  x\neq 1$,
        \item $xR_{\top\limp(q\land p)}y : \iff x\leq y$,
        \item $xR_Ay$ holds for all $x, y$ and all other formulas $A$,
        \item $V(p)=\{2\}$ and $V(q)=\emptyset$.
    \end{itemize}

    The relations $R_{\top\limp(p\land q)}$ and $R_{\top\limp(q\land p)}$ are visualized in Figure \ref{Fig:VF_vs_WF}.

    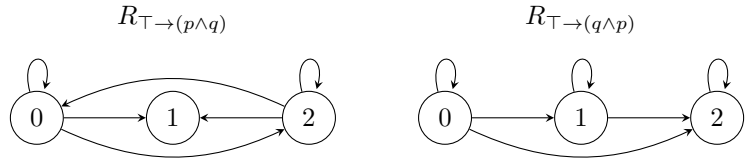
\begin{figure}[ht]
        \centering
        \begin{tikzpicture}[>=stealth, node distance=18mm]
            \tikzstyle{world}=[circle, draw, minimum size=7mm, inner sep=0pt]

            \node at (1.8,1.3) {$R_{\top\limp(p\land q)}$};

            \node[world] (a0) at (0,0) {$0$};
            \node[world] (a1) at (1.8,0) {$1$};
            \node[world] (a2) at (3.6,0) {$2$};

            \draw[->] (a0) edge[loop above] (a0);
            \draw[->] (a0) edge (a1);
            \draw[->] (a0) edge[bend right=25] (a2);

            \draw[->] (a2) edge[bend right=25] (a0);
            \draw[->] (a2) edge (a1);
            \draw[->] (a2) edge[loop above] (a2);

            \node at (7.2,1.3) {$R_{\top\limp(q\land p)}$};

            \node[world] (b0) at (5.4,0) {$0$};
            \node[world] (b1) at (7.2,0) {$1$};
            \node[world] (b2) at (9.0,0) {$2$};

            \draw[->] (b0) edge[loop above] (b0);
            \draw[->] (b1) edge[loop above] (b1);
            \draw[->] (b2) edge[loop above] (b2);

            \draw[->] (b0) edge (b1);
            \draw[->] (b1) edge (b2);
            \draw[->] (b0) edge[bend right=25] (b2);

        \end{tikzpicture}
        \caption{The two specific relations of $M$}\label{Fig:VF_vs_WF}
    \end{figure}

    Since $M,2\nforces p\land q$ and $1R_{\top\limp(p\land q)}2$, we have
    $M,1\nforces\top\limp(p\land q)$.
    On the other hand, since there is no $x$ such that $1R_{\top\limp(q\land p)}x$, we have $M,1\forces\top\limp(q\land p)$.
    Therefore, since $0$ is a root of $M$, we obtain
    \[
        M,0\nforces
        (\top\limp(q\land p))\limp(\top\limp(p\land q)).
    \]
    Thus $M,0\nforces(\top\limp(p\land q))\liff(\top\limp(q\land p))$.
    By soundness, $\LogicVF$ does not prove this formula.
\end{proof}

\subsection{Frame conditions of closed negative axioms}
\label{subsect:frame_condition_FMT}

In this subsection, we analyze frame conditions of FMT frames corresponding to closed negative axioms.
Since closed formulas contain no propositional variables, their validity on an FMT frame is independent of valuations.
So, for a closed formula $A$, whether $x\forces A$ holds depends only on the FMT frame.

\begin{lem}\label{lem:closed}
    Let $A$ be a closed formula.
    If $M$ and $N$ are FMT models based on the same FMT frame $F$, then for every $x \in F$, we have $M,x\forces A$ if and only if $N,x\forces A$.
\end{lem}

We investigate the behavior of closed negative formulas.

\begin{prop}\label{prop:closed_negative}
    Let $C$ and $D$ be closed formulas, and let $F$ be an FMT frame.
    \begin{enumerate}
        \item $F\models\neg C$ if and only if there is no $x \in F$ such that $x\forces C$.

        \item $F\models\neg(C\lor D)$ if and only if $F\models\neg C$ and
              $F\models\neg D$.

        \item $F\models\neg(C\limp D)$ if and only if for every world $x$ of $F$,
              there exists a world $y$ such that $xR_{C\limp D}y$, $y\forces C$, and
              $y\nforces D$.

        \item $F\models\neg\neg C$ if and only if for every world $x$ of $F$,
              there exists a world $y$ such that $xR_{\neg C}y$ and $y\forces C$.
    \end{enumerate}
\end{prop}

\begin{proof}
    1. $(\Rightarrow)$:
    Suppose $F\models\neg C$.
    Let $x \in F$ be arbitrary world.
    For the root $r$ of $F$, we have $r\forces\neg C$.
    Since $rR_{\neg C}x$, we obtain $x \nVdash C$.

    \medskip

    $(\Leftarrow)$: Suppose that $x \nVdash C$ for all $x \in F$.
    Then, for every $x, y \in F$, if $xR_{\neg C}y$, then $y\nforces C$.
    Hence $x\forces\neg C$ for all $x \in F$.
    Therefore $F\models\neg C$.

    \medskip

    2. By (1), $F\models\neg(C\lor D)$ if and only if $x \nVdash C\lor D$ for all $x \in F$.
    This is equivalent to $x \nVdash C$ and $x \nVdash D$ for all $x \in F$.
    By (1) again, this is equivalent to $F\models\neg C$ and $F\models\neg D$.

    \medskip

    3. By (1), $F\models\neg(C\limp D)$ if and only if $x \nVdash C\limp D$ for all $x \in F$.
    For $x \in W$, $x \nVdash C\limp D$ if and only if there exists $y \in F$ such that $xR_{C\limp D}y$, $y\forces C$, and $y\nforces D$.
    This proves the required equivalence.

    \medskip

    4. Immediate from (3).
\end{proof}

As concrete examples, the frame conditions corresponding to some closed
negative axioms are described as follows.
We say that an FMT frame $F=\langle W,\{R_A\}_{A\in\PFml},r\rangle$ is
\emph{$A$-serial} if for every $x\in W$ there exists $y\in W$ such that
$xR_Ay$.

\begin{cor}\label{cor:axioms}
    Let $F$ be any FMT frame.
    \begin{enumerate}
        \item $F\models\neg \neg\top$ if and only if $F$ is
              $\neg\top$-serial.

        \item $F\models\neg(\top\limp\neg\top)$ if and only if for any $x \in F$, there exist $y, z \in W$ such that $xR_{\top\limp\neg\top}y$ and $yR_{\neg\top}z$.
    \end{enumerate}
\end{cor}

We now define the class of frames corresponding to closed negative extensions.
Let $X$ be a finite set of closed negative axioms and let $L=\LogicVF+X$.
An FMT frame $F$ is called an \emph{$L$-frame} if $F\models B$ for every
$B\in X$.

\begin{thm}[Soundness for closed negative extensions]\label{thmr:L_sound}
    Let $X$ be a finite set of closed negative axioms and let $L=\LogicVF+X$.
    For any $A\in\PFml$ and any $L$-frame $F$, if $L\vdash A$, then
    $F\models A$.
\end{thm}

\begin{proof}
    The proof is by induction on the length of an $L$-proof of $A$.
    The axioms and rules of $\LogicVF$ are sound by the soundness of $\LogicVF$.
    If a formula occurring in the proof is an additional axiom $B\in X$, then
    $F\models B$ by the definition of $L$-frames.
    Therefore every formula occurring in the proof is valid on $F$.
\end{proof}

\section{Finite frame property for closed negative extensions}\label{sec:FFP}

In this section, we prove the finite frame property not only for $\LogicVF$,
but also for its closed negative extensions.
We fix a finite set $X$ of closed negative axioms, and let $L=\LogicVF+X$.

Let $A$ be a formula.
We define
\[
    \Subfml_X(A):=\Subfml(A)\cup\bigcup_{B\in X}\Subfml(B) \cup \{\bot, \top\}.
\]
Since $X$ is finite, $\Subfml_X(A)$ is finite.

\begin{defn}
    Let $\Gamma,\Delta$ be finite subsets of $\Subfml_X(A)$.  We call a pair
    $(\Gamma,\Delta)$ an \emph{$(A,X)$-tableau}.
    \begin{itemize}
        \item We write $(\Gamma,\Delta)\subseteq(\Gamma',\Delta')$ for
              $\Gamma\subseteq\Gamma'$ and $\Delta\subseteq\Delta'$.
        \item $(\Gamma,\Delta)$ is \emph{$L$-consistent} if $L\nvdash \lconj\Gamma\limp\ldisj\Delta$.
        \item $(\Gamma,\Delta)$ is \emph{saturated} if $\Gamma\cup\Delta=\Subfml_X(A)$.
    \end{itemize}
\end{defn}

\begin{lem}\label{lem:either_L_consistent}
    Let $(\Gamma,\Delta)$ be an $L$-consistent $(A,X)$-tableau and
    $B\in\Subfml_X(A)$.  Then at least one of
    $(\Gamma\cup\{B\},\Delta)$ and $(\Gamma,\Delta\cup\{B\})$ is $L$-consistent.
\end{lem}

\begin{proof}
    Suppose, towards a contradiction, that both are not $L$-consistent.  Then
    $L\vdash (\lconj\Gamma\land B)\limp\ldisj\Delta$ and
    $L\vdash \lconj\Gamma\limp(\ldisj\Delta\lor B)$.  By $\RuleRD$, from the
    first formula and $L\vdash\ldisj\Delta\limp\ldisj\Delta$, we get
    $L\vdash \ldisj\Delta\lor(\lconj\Gamma\land B)\limp\ldisj\Delta$.

    On the other hand, by $\RuleRC$, from
    $L\vdash \lconj\Gamma\limp(\ldisj\Delta\lor B)$ and
    $L\vdash \lconj\Gamma\limp(\ldisj\Delta\lor\lconj\Gamma)$, we get
    $L\vdash \lconj\Gamma\limp
        (\ldisj\Delta\lor\lconj\Gamma)\land(\ldisj\Delta\lor B)$.
    By Proposition \ref{prop:VF_aux_theorems}, we have $L\vdash
        (\ldisj\Delta\lor\lconj\Gamma)\land(\ldisj\Delta\lor B)
        \limp \ldisj\Delta\lor(\lconj\Gamma\land B)$.
    Then, two applications of $\RuleRI$ yield $L\vdash\lconj\Gamma\limp\ldisj\Delta$.
    This contradicts the $L$-consistency of $(\Gamma,\Delta)$.
\end{proof}

\begin{lem}[Lindenbaum Lemma]\label{lem:L_lindenbaum}
    Every $L$-consistent $(A,X)$-tableau $(\Gamma,\Delta)$ can be expanded to a
    saturated $L$-consistent $(A,X)$-tableau $(\Gamma',\Delta')\supseteq
        (\Gamma,\Delta)$.
\end{lem}

\begin{proof}
    Let $B_1,\ldots,B_n$ be an enumeration of $\Subfml_X(A)$.
    We define $(A,X)$-tableaux $(\Gamma_i,\Delta_i)$ for $0\leq i\leq n$ as follows:
    $(\Gamma_0,\Delta_0):=(\Gamma,\Delta)$, and
    \[
        (\Gamma_{i+1},\Delta_{i+1}):=
        \begin{cases}
            (\Gamma_i\cup\{B_{i+1}\},\Delta_i) &
            \text{if this tableau is $L$-consistent},              \\
            (\Gamma_i,\Delta_i\cup\{B_{i+1}\}) & \text{otherwise.}
        \end{cases}
    \]
    By Lemma \ref{lem:either_L_consistent}, each $(\Gamma_i,\Delta_i)$ is
    $L$-consistent.
    The final tableau $(\Gamma_n,\Delta_n)$ is saturated and
    extends $(\Gamma,\Delta)$.
\end{proof}

\begin{lem}\label{lem:L_SCT_property}
    Let $(\Gamma,\Delta)$ be a saturated $L$-consistent $(A,X)$-tableau.
    \begin{enumerate}
        \item If $B,C\in\Subfml_X(A)$, $L\vdash B\limp C$, and $B\in\Gamma$,
              then $C\in\Gamma$.
        \item $\bot\in\Delta$.
        \item $\top\in\Gamma$.
        \item If $B\land C\in\Subfml_X(A)$, then
              $B\land C\in\Gamma$ if and only if $B\in\Gamma$ and $C\in\Gamma$.
        \item If $B\lor C\in\Subfml_X(A)$, then
              $B\lor C\in\Gamma$ if and only if $B\in\Gamma$ or $C\in\Gamma$.
        \item If $B\in\Subfml_X(A)$ and $L\vdash B$, then $B\in\Gamma$.
    \end{enumerate}
\end{lem}

\begin{proof}
    1. Suppose $L\vdash B\limp C$ and $B\in\Gamma$.
    If $C\in\Delta$, then $L\vdash \lconj\Gamma\limp B$ and $L\vdash C\limp\ldisj\Delta$.
    Hence, by $\RuleRI$, $L\vdash \lconj\Gamma\limp\ldisj\Delta$, contradicting the $L$-consistency of $(\Gamma,\Delta)$.
    Since the tableau is saturated, we obtain $C\in\Gamma$.

    \medskip

    2. If $\bot\in\Gamma$, then $L\vdash\lconj\Gamma\limp\ldisj\Delta$ by the
    axiom $\bot\limp \ldisj\Delta$, a contradiction.
    Hence $\bot\in\Delta$ by saturation.

    \medskip

    3. If $\top\in\Delta$, then $L\vdash\lconj\Gamma\limp\ldisj\Delta$ because $L\vdash \top$.
    This is a contradiction.
    Hence $\top\in\Gamma$ by saturation.

    \medskip

    4. $(\Rightarrow)$: This direction follows from (1), using
    $L\vdash B\land C\limp B$ and $L\vdash B\land C\limp C$.

    \medskip

    $(\Leftarrow)$:
    If $B,C\in\Gamma$ and $B\land C\in\Delta$, then
    $L\vdash \lconj\Gamma\limp B$ and $L\vdash \lconj\Gamma\limp C$ implies $L\vdash \lconj\Gamma\limp B\land C$, contradicting the $L$-consistency.  Hence $B\land C\in\Gamma$.

    \medskip

    5. The proof is similar, using the axioms for disjunction and $\RuleRD$.

    \medskip

    6. Suppose $L\vdash B$.  By $\RuleAF$, $L\vdash \top\limp B$.  By (3),
    $\top\in\Gamma$, and by (1), $B\in\Gamma$.
\end{proof}

Since $\Subfml_X(A)$ is finite, the set of all saturated $L$-consistent
$(A,X)$-tableaux is finite.
Let $W_A^X$ be the set of all saturated $L$-consistent $(A,X)$-tableaux.
We define an $(A,X)$-tableau $(\Gamma_r^0,\Delta_r^0)$ by
\begin{itemize}
    \item $\Gamma_r^0:=\emptyset$,
    \item $\Delta_r^0:=
              \{\,B\limp C\in\Subfml_X(A)\mid B\in\Gamma$ and $C\in\Delta$ for some $(\Gamma,\Delta)\in W_A^X\,\}$.
\end{itemize}

\begin{lem}\label{lem:L_root}
    The tableau $(\Gamma_r^0,\Delta_r^0)$ is $L$-consistent.
\end{lem}

\begin{proof}
    Suppose that $(\Gamma_r^0,\Delta_r^0)$ is not $L$-consistent.
    Then $L\vdash \top\limp\ldisj\Delta_r^0$.
    Since $L\vdash\top$, we obtain $L\vdash\ldisj\Delta_r^0$ by $\RuleMP$.

    If $\Delta_r^0$ is empty, then $L\vdash\bot$, contradicting the consistency
    of $L$.
    Thus $\Delta_r^0$ is nonempty.
    Since $L$ has DP (Theorem \ref{thm:VF_has_DP}), there is $B\limp C\in\Delta_r^0$ such that $L\vdash B\limp C$.

    By the definition of $\Delta_r^0$, there is $(\Gamma,\Delta)\in W_A^X$ such that $B\in\Gamma$ and $C\in\Delta$.
    Then $L\vdash\lconj\Gamma\limp B$ and $L\vdash C\limp\ldisj\Delta$.  Hence, by two applications of $\RuleRI$, we obtain $L\vdash\lconj\Gamma\limp\ldisj\Delta$, contradicting the $L$-consistency of $(\Gamma,\Delta)$.
\end{proof}

By Lemma \ref{lem:L_root} and the Lindenbaum
Lemma, we fix a saturated $L$-consistent $(A,X)$-tableau $(\Gamma_r,\Delta_r)$ extending $(\Gamma_r^0,\Delta_r^0)$.

\begin{defn}
    We define a tuple $M^X_A=\langle W_A^X,\{R_B\}_{B\in\PFml},(\Gamma_r,\Delta_r),V\rangle$ as follows:
    \begin{itemize}
        \item $R_B$ is defined as follows:
              \begin{itemize}
                  \item
                        If $B \equiv C \limp D$ and $C \limp D \in \Subfml_X(A)$,
                        then
                        \[
                            (\Gamma_1, \Delta_1) R_B (\Gamma_2, \Delta_2) :\iff C \limp D \in \Delta_1,\ \text{or}\ C \in \Delta_2,\ \text{or}\ D \in \Gamma_2.
                        \]
                  \item
                        Otherwise, $(\Gamma_1, \Delta_1) R_B (\Gamma_2, \Delta_2)$ always holds.
              \end{itemize}

        \item $(\Gamma,\Delta)\in V(p) : \iff p\in\Gamma$.
    \end{itemize}
\end{defn}

\begin{lem}\label{lem:L_countermodel_is_FMT_model}
    The structure $M_A^X$ is a finite FMT model.
\end{lem}

\begin{proof}
    It suffices to show that $(\Gamma_r,\Delta_r)$ is a root.
    Let $(\Gamma,\Delta)\in W_A^X$ and $B\in\PFml$ be arbitrary.
    If $B$ is not of the form $C\limp D$ with $C\limp D\in\Subfml_X(A)$, then
    $(\Gamma_r,\Delta_r)R_B(\Gamma,\Delta)$ holds by the definition of $R_{B}$.

    Suppose $B\equiv C\limp D$ and $C\limp D\in\Subfml_X(A)$.
    If $C\in\Delta$ or $D\in\Gamma$, then $(\Gamma_r,\Delta_r)R_{C\limp D}(\Gamma,\Delta)$ by the definition of $R_{C \limp D}$.
    Otherwise, by saturation, $C\in\Gamma$ and $D\in\Delta$.
    Hence $C\limp D\in\Delta_r^0\subseteq\Delta_r$, and so $(\Gamma_r,\Delta_r)R_{C\limp D}(\Gamma,\Delta)$ by the definition of $R_{C \limp D}$.
    We have proved that $(\Gamma_r,\Delta_r)$ is a root.
\end{proof}

\begin{lem}[Truth Lemma]\label{lem:L_truthlemma}
    Let $B\in\Subfml_X(A)$ and $(\Gamma,\Delta)\in W_A^X$.
    Then
    \[
        M_A^X,(\Gamma,\Delta)\forces B\ \text{if and only if}\ B\in\Gamma.
    \]
\end{lem}

\begin{proof}
    We prove the lemma by induction on the construction of $B$.
    The case where $B$ is a propositional variable follows from the definition of $V$.
    The case $B\equiv\bot$ follows from Lemma \ref{lem:L_SCT_property} (2).
    The cases for conjunction and disjunction follow from Lemma \ref{lem:L_SCT_property} (4) and (5), respectively.
    So, it suffices to prove the case $B\equiv C\limp D$.

    \medskip

    $(\Rightarrow)$: We prove the contraposition.
    Suppose $C\limp D\notin\Gamma$.
    Since $(\Gamma,\Delta)$ is saturated, $C\limp D\in\Delta$.
    We show that $(\{C\},\{D\})$ is $L$-consistent.
    If not, then $L\vdash C\limp D$.
    Since $C\limp D\in\Subfml_X(A)$, it follows $C\limp D\in\Gamma$ from Lemma \ref{lem:L_SCT_property} (6), a contradiction.
    Thus $(\{C\},\{D\})$ is $L$-consistent.
    By Lemma \ref{lem:L_lindenbaum}, there is a saturated $L$-consistent $(A,X)$-tableau $(\Sigma,\Pi)$ extending $(\{C\},\{D\})$.
    Then $(\Sigma,\Pi)\in W_A^X$, $C\in\Sigma$, and $D\in\Pi$.
    By the induction hypothesis, we obtain $(\Sigma,\Pi)\forces C$ and $(\Sigma,\Pi)\nforces D$.
    Since $C\limp D\in\Delta$, we have $(\Gamma,\Delta)R_{C\limp D}(\Sigma,\Pi)$.
    Therefore $(\Gamma,\Delta)\nforces C\limp D$.

    \medskip

    $(\Leftarrow)$:
    Suppose $C\limp D\in\Gamma$.
    Let $(\Sigma,\Pi)\in W_A^X$ be arbitrary tableau such that $(\Gamma,\Delta)R_{C\limp D}(\Sigma,\Pi)$, and suppose $(\Sigma,\Pi)\forces C$.
    By the induction hypothesis, $C\in\Sigma$.
    By the definition of $R_{C\limp D}$, either $C\limp D\in\Delta$, or
    $C\in\Pi$, or $D\in\Sigma$.
    The first case violates the $L$-consistency of $(\Gamma,\Delta)$, because $C\limp D\in\Gamma$.
    The second case violates the $L$-consistency of $(\Sigma,\Pi)$, because $C\in\Sigma$.
    Hence $D\in\Sigma$.
    By the induction hypothesis, $(\Sigma,\Pi)\forces D$.
    Thus $(\Gamma,\Delta)\forces C\limp D$.
\end{proof}

\begin{lem}\label{lem:L_frame}
    The frame of $M_A^X$ is an $L$-frame.
\end{lem}

\begin{proof}
    Let $E\in X$.  Since $E\in\Subfml_X(A)$ and $L\vdash E$, it follows from Lemma \ref{lem:L_SCT_property} (6) that  $E\in\Gamma$ for every
    $(\Gamma,\Delta)\in W_A^X$.
    By Lemma \ref{lem:L_truthlemma}, $(\Gamma, \Delta) \forces E$ for every $(\Gamma, \Delta) \in W_A^X$.
    Hence the frame of $M_A^X$ validates every element of $X$ by Lemma \ref{lem:closed}.
    Therefore it is an $L$-frame.
\end{proof}

\begin{thm}[Finite frame property for closed negative extensions]\label{thm:finite_frames}
    Let $X$ be a finite set of closed negative axioms and let $L=\LogicVF+X$.
    For any $A\in\PFml$, the following are equivalent:
    \begin{enumerate}
        \item $L\vdash A$.
        \item $F\models A$ for all $L$-frames $F$.
        \item $F\models A$ for all finite $L$-frames $F$.
    \end{enumerate}
\end{thm}

\begin{proof}
    $(1\Rightarrow 2)$ follows from the soundness theorem for closed negative extensions (Theorem \ref{thmr:L_sound}).
    The implication $(2\Rightarrow 3)$ is immediate.

    We prove the contraposition of $(3\Rightarrow 1)$.
    Suppose $L\nvdash A$.
    Then $(\emptyset,\{A\})$ is an $L$-consistent $(A,X)$-tableau.
    By Lemma \ref{lem:L_lindenbaum}, there is a saturated $L$-consistent $(A,X)$-tableau
    $(\Gamma_A,\Delta_A)$ extending $(\emptyset,\{A\})$.
    Then $(\Gamma_A,\Delta_A)\in W_A^X$ and $A\in\Delta_A$.
    By Lemma \ref{lem:L_truthlemma}, we have
    $M_A^X,(\Gamma_A,\Delta_A)\nforces A$.
    By Lemma \ref{lem:L_frame}, the frame of $M_A^X$ is a finite $L$-frame.
    Hence there exists a finite $L$-frame which does not validate $A$.
\end{proof}

For example, we obtain the following characterization results.

\begin{cor}\label{cor:VF}
    For any $A\in\PFml$, the following are equivalent:
    \begin{enumerate}
        \item $\LogicVF\vdash A$.
        \item $F\models A$ for all FMT frames $F$.
        \item $F\models A$ for all finite FMT frames $F$.
    \end{enumerate}
\end{cor}

\begin{cor}\label{cor:VF_notnot_top}
    For any $A\in\PFml$, the following are equivalent:
    \begin{enumerate}
        \item $\LogicVF+\neg\neg\top\vdash A$.
        \item $F\models A$ for all $\neg\top$-serial frames $F$.
        \item $F\models A$ for all finite $\neg\top$-serial frames $F$.
    \end{enumerate}
\end{cor}

\section{Modal companions} \label{sec:MC}

In this section, we study modal companions of $\LogicVF$ and its closed
negative extensions via Corsi's modified G{\"o}del translation.
In Subsection \ref{ssec:HB}, we recall some background on modal companions of intuitionistic and subintuitionistic logics.
In Subsection \ref{ssec:MCS}, we show that closed negative extensions of $\LogicVF$ correspond to suitable closed modal extensions of the pure logic of necessitation $\Logic{N}$.
In particular, we prove that $\Logic{N}$ is one of the modal companions of $\LogicVF$.

\subsection{Historical background}\label{ssec:HB}

In this subsection, we recall some standard background on modal companions of intuitionistic and subintuitionistic logics.

First of all, the celebrated result of G{\"o}del \cite{God33} and
McKinsey--Tarski \cite{MT48} states that intuitionistic logic
$\Logic{Int}$ can be embedded into modal logic $\Logic{S4}$ via the following G\"odel translation.
Let $\MFml$ be the set of all modal formulas.

\begin{defn}[G\"odel translation]
    G{\"o}del translation $(\cdot)^\goedelTr : \PFml \to \MFml$ is defined as follows:
    \begin{enumerate}
        \item $p^\goedelTr = \Box p$ where $p \in \PropVar$,
        \item $\bot^\goedelTr = \bot$,
        \item $(A \land B)^\goedelTr = (A^\goedelTr \land B^\goedelTr)$,
        \item $(A \lor B)^\goedelTr = (A^\goedelTr \lor B^\goedelTr)$,
        \item $(A \limp B)^\goedelTr = \Box(A^\goedelTr \limp B^\goedelTr)$.
    \end{enumerate}
\end{defn}

\begin{fact}[G\"odel--McKinsey--Tarski's theorem \cite{God33, MT48}]
    For any $A \in \PFml$, $\Logic{Int} \proves A$ if and only if $\Logic{S4} \proves A^\goedelTr$.
\end{fact}

This embedding phenomenon is expressed by saying that $\Logic{S4}$ is a \emph{modal companion} of $\Logic{Int}$.
The key to this proof via Kripke semantics is that the corresponding models for both $\Logic{Int}$ and $\Logic{S4}$ are preorders (see \cite[Theorem~3.83]{CZ97}).

Corsi \cite{Cor87} introduced subintuitionistic logics, that is, propositional logics corresponding to models that do not necessarily satisfy the preorder condition and persistency, both of which are required for $\Logic{Int}$.
Corsi also showed that these logics have modal logics weaker than
$\Logic{S4}$ as their modal companions, such as $\Logic{K}$, $\Logic{KT}$,
and $\Logic{K4}$.

When dealing with subintuitionistic logics, the original G{\"o}del translation is no longer the most suitable one.
The reason is that their semantics, such as Kripke, neighborhood, and FMT semantics, do not always ensure persistency for propositional variables.
For this reason, Corsi used a modified G{\"o}del translation in which propositional variables are translated as themselves.

\begin{defn}
    Corsi's modified G{\"o}del translation $(\cdot)^\corsiTr : \PFml \to \MFml$ is defined as follows:
    \begin{enumerate}
        \item $p^\corsiTr = p$ for $p \in \PropVar$,
        \item $\bot^\corsiTr = \bot$,
        \item $(A \land B)^\corsiTr = A^\corsiTr \land B^\corsiTr$,
        \item $(A \lor B)^\corsiTr = A^\corsiTr \lor B^\corsiTr$,
        \item $(A \limp B)^\corsiTr = \Box(A^\corsiTr \limp B^\corsiTr)$.
    \end{enumerate}
\end{defn}

\begin{fact}[Corsi \cite{Cor87}]
    Let $A \in \PFml$.
    \begin{enumerate}
        \item $\LogicF \proves A$ if and only if $\Logic{K} \proves A^\corsiTr$,
        \item $\LogicF + \neg \neg \top \proves A$ if and only if $\Logic{KD} \proves A^\corsiTr$.
    \end{enumerate}
\end{fact}

In contrast to Corsi's relational treatment of $\LogicF$, the weak
subintuitionistic logic $\LogicWF$ of Maleki and de Jongh is naturally connected with neighborhood semantics.
On the modal side, this corresponds to the non-normal modal logic $\Logic{E}$ and its extensions, which are weaker than $\Logic{K}$ (see \cite{Che80, Pac17}).
De Jongh and Maleki \cite{dJM18,dJM19} showed that suitable extensions of $\LogicWF$ have extensions of $\Logic{E}$ as their modal companions via Corsi's modified G{\"o}del translation.







Thus, while $\LogicF$ is linked to normal modal logics such as $\Logic{K}$ and $\Logic{KD}$, and $\LogicWF$ is linked to neighborhood-based non-normal modal logics such as $\Logic{E}$, we shall show in the next subsection that our still weaker logic $\LogicVF$ is linked to the pure logic of necessitation $\Logic{N}$.

\subsection{Modal companions of $\LogicVF$ and its closed negative extensions}\label{ssec:MCS}

We now turn to our logic $\LogicVF$ and its closed negative extensions.
In this subsection, we show that they correspond, via Corsi's modified
G{\"o}del translation, to the pure logic of necessitation $\Logic{N}$ and its closed modal extensions.

We first recall the modal logic $\Logic{N}$ and its FMT semantics.
The logic $\Logic{N}$ is the modal logic obtained by adding the necessitation rule $\dfrac{A}{\Box A}$ to classical propositional logic.
A modal formula is said to be \emph{closed} if it contains no propositional
variables.
For each finite set $Y$ of closed modal formulas, let $\Logic{N} + Y$ denote the modal logic obtained by adding all formulas in $Y$ as
axioms to $\Logic{N}$.

\begin{defn}[Modal FMT frames and models]
    A \emph{modal FMT frame} is a pair $F=\langle W,\{R_A\}_{A\in\MFml}\rangle$ such that $W$ is a non-empty set and
    $R_A$ is a binary relation on $W$ for each modal formula $A\in\MFml$.

    A \emph{modal FMT model} is a pair $M=\langle F,V\rangle$, where $F$ is a
    modal FMT frame and $V\colon W\to 2^{\PropVar}$ is a valuation.
    The satisfaction relation is defined as usual for Boolean connectives, and the modal clause is:
    \begin{itemize}
        \item $M,x\forces \Box A : \iff$ for every $y\in W$, if $xR_Ay$, then
              $M,y\forces A$.
    \end{itemize}

    A modal formula $A$ is said to be \emph{valid} in a modal FMT model $M$ (written $M\models A$) if $M,x\forces A$ for all $x\in W$.
    A modal formula $A$ is \emph{valid} in a modal FMT frame $F$ (written $F\models A$) if $A$ is valid in all modal FMT models based on $F$.
\end{defn}

\begin{fact}[Fitting, Marek and Truszczy\'nski \cite{FMT92}]
    The logic $\Logic{N}$ is sound and complete with respect to the class of all modal FMT frames.
    Moreover, $\Logic{N}$ has the finite frame property with respect to this semantics.
\end{fact}

We shall also use the corresponding finite frame property for closed modal
extensions of $\Logic{N}$.  Since the proof is a straightforward adaptation of
the usual finite frame construction for $\Logic{N}$, we give the details in
Appendix.

\begin{thm}\label{thm:N_FFP}
    Let $Y$ be a finite set of closed modal formulas and let $L = \Logic{N}+Y$.
    For any modal formula $A$, the following are equivalent:
    \begin{enumerate}
        \item $L\vdash A$.
        \item $F\models A$ for all modal FMT frames $F$ validating every formula in $Y$.
        \item $F\models A$ for all finite modal FMT frames $F$ validating every
              formula in $Y$.
    \end{enumerate}
\end{thm}

We relate propositional FMT models and modal FMT models via Corsi's
modified G{\"o}del translation.
For a finite set $X$ of closed negative axioms, we define $X^\ast :=\{\neg C^\corsiTr\mid \neg C\in X\}$.
Since every formula in $X$ is closed, every formula in $X^\ast$ is a
closed modal formula.
So, $\Logic{N}+X^\ast$ is a closed modal extension of $\Logic{N}$.

\begin{lem}\label{lem:prop_to_modal_FMT}
    Let $M^\mathrm{P}=\langle W,\{R^\mathrm{P}_A\}_{A\in\PFml},r,V \rangle$ be a propositional FMT model.
    We define a modal FMT model $M^\mathrm{M}=\langle W,\{R^\mathrm{M}_B\}_{B\in\MFml},V \rangle$ as follows.
    \begin{itemize}
        \item If $B = (C^\corsiTr\limp D^\corsiTr)$ for some
              $C,D\in\PFml$, then $xR^\mathrm{M}_By : \iff xR^\mathrm{P}_{C\limp D}y$.
        \item Otherwise, $xR^\mathrm{M}_By$ always holds.
    \end{itemize}
    Then, for every $A\in\PFml$ and every $x\in W$,
    \[
        M^\mathrm{P},x\forces A\ \text{if and only if}\ M^\mathrm{M},x\forces A^\corsiTr.
    \]
\end{lem}

\begin{proof}
    We prove the lemma by induction on the construction of $A$.
    The cases of propositional variables, $\bot$, conjunction, and disjunction are immediate.

    Suppose $A\equiv (C\limp D)$.
    By the induction hypothesis, for every $x\in W$, $M^\mathrm{P},x\forces C$ if and only if $M^\mathrm{M},x\forces C^\corsiTr$,
    and similarly for $D$.
    Also, $(C\limp D)^\corsiTr=\Box(C^\corsiTr\limp D^\corsiTr)$.
    \begin{align*}
        M^\mathrm{P},x \forces C \to D & \iff \forall y (x R^{\mathrm{P}}_{C \to D} y\ \&\ M^\mathrm{P},y \forces C \Rightarrow M^\mathrm{P},y \forces D)                                     \\
                                       & \iff \forall y (x R^{\mathrm{P}}_{C \to D} y\ \&\ M^\mathrm{P},y \forces C \Rightarrow M^\mathrm{P},y \forces D)                                     \\
                                       & \iff \forall y (x R^{\mathrm{M}}_{C^\corsiTr \to D^\corsiTr} y\ \&\ M^\mathrm{M},y \forces C^\corsiTr \Rightarrow M^\mathrm{M},y \forces D^\corsiTr) \\
                                       & \iff \forall y (x R^{\mathrm{M}}_{C^\corsiTr \to D^\corsiTr} y\ \&\ M^\mathrm{M},y \forces C^\corsiTr\to D^\corsiTr)                                 \\
                                       & \iff M^\mathrm{M},x \forces \Box (C^\corsiTr\to D^\corsiTr)                                                                                          \\
                                       & \iff M^\mathrm{M},x \forces (C \to D)^\corsiTr. \qedhere
    \end{align*}
\end{proof}

We also use the following converse construction.

\begin{lem}\label{lem:modal_to_prop_FMT}
    Let $M^\mathrm{M}=\langle W,\{R^\mathrm{M}_B\}_{B\in\MFml},V\rangle$ be a
    modal FMT model.
    Take a fresh point $r\notin W$.
    We define a tuple $M^\mathrm{P}=\langle W\cup\{r\},\{R^\mathrm{P}_A\}_{A\in\PFml},r,V^\ast\rangle$ as follows.
    \begin{itemize}
        \item For $x\in W$, let $V^\ast(x)=V(x)$, and let $V^\ast(r)=\PropVar$.
        \item If $A\equiv (C\limp D)$, then $xR^\mathrm{P}_Ay : \iff x=r$ or $(x,y\in W$ and $xR^\mathrm{M}_{C^\corsiTr\limp D^\corsiTr}y)$.
        \item Otherwise, $xR^\mathrm{P}_Ay : \iff$ either $x=r$ or $x, y\in W$.
    \end{itemize}
    Then $M^\mathrm{P}$ is a propositional FMT model.
    Moreover, for every $A\in\PFml$ and every $x\in W$,
    \[
        M^\mathrm{P},x\forces A\ \text{if and only if}\ M^\mathrm{M},x\forces A^\corsiTr.
    \]
\end{lem}

\begin{proof}
    By the definition, $rR^\mathrm{P}_Ay$ holds for every $A\in\PFml$ and every
    $y\in W\cup\{r\}$.
    Hence $r$ is a root, and so $M^\mathrm{P}$ is a propositional FMT model.

    The equivalence is proved by induction on the construction of $A$, in the
    same way as in Lemma \ref{lem:prop_to_modal_FMT}.
\end{proof}

\begin{lem}\label{lem:modal_to_prop_closed_negative}
    Let $C$ be a closed propositional formula and let $M^\mathrm{M}$ and
    $M^\mathrm{P}$ be as in Lemma \ref{lem:modal_to_prop_FMT}.  If
    $M^\mathrm{M}\models\neg C^\corsiTr$, then for all $x \in W \cup \{r\}$, we have $M^\mathrm{P}, x \nVdash C$.
\end{lem}

\begin{proof}
    Since $M^{\mathrm{M}}, x \nVdash C^\corsiTr$ for all $x \in W$, by Lemma \ref{lem:modal_to_prop_FMT}, $M^{\mathrm{P}}, x \nVdash C$ for all $x \in W$.

    \begin{cl}
        For any closed formula $C$, if $M^\mathrm{P},r\forces C$, then $M^\mathrm{M}\models C^\corsiTr$.
    \end{cl}
    \begin{proof}[Proof of Claim]
        We prove the claim by inducion on the construction of closed $C$.
        The cases of $\bot$, conjunction, and disjunction are immediate.

        Suppose $C\equiv (D\limp E)$ and $M^\mathrm{P},r\forces D\limp E$.
        Let $y \in W$ be arbitrary element with $M^\mathrm{M},y\forces D^\corsiTr$.
        By Lemma \ref{lem:modal_to_prop_FMT}, we have $M^\mathrm{P},y\forces D$.
        Since $rR^\mathrm{P}_{D\limp E}y$, we have $M^\mathrm{P},y\forces E$.
        By Lemma \ref{lem:modal_to_prop_FMT} again, we obtain $M^\mathrm{M},y\forces E^\corsiTr$.
        We have proved that $M^\mathrm{M}\models D^\corsiTr\limp E^\corsiTr$, and so $M^\mathrm{M}\models \Box(D^\corsiTr\limp E^\corsiTr)$.
        That is, $M^\mathrm{M}\models C^\corsiTr$.
    \end{proof}

    If $M^\mathrm{P},r\forces C$, then it follow from the claim that $M^\mathrm{M}\models C^\corsiTr$, contradicting $M^\mathrm{M}\models\neg C^\corsiTr$.
    Thus $M^\mathrm{P}, r\nforces C$.
\end{proof}

For each $k \geq 0$, let $\neg^k$ be the expression $\overbrace{\neg \neg \cdots \neg}^{k}$.

\begin{thm}\label{thm:modal_companion}
    Let $X$ be a finite set of closed negative axioms.  For any $A\in\PFml$, the
    following are equivalent:
    \begin{enumerate}
        \item $\LogicVF+X\vdash A$.
        \item $\Logic{N}+X^\ast\vdash A^\corsiTr$.
        \item $\Logic{N}+X^\ast+ \{\neg\Box \neg^{2k} \bot \mid k \geq 0\} \vdash A^\corsiTr$.
    \end{enumerate}
\end{thm}

\begin{proof}
    Let $L^{\mathrm{P}}=\LogicVF+X$ and $L^{\mathrm{M}}=\Logic{N}+X^\ast$.

    $(1\Rightarrow 2)$: We prove the contraposition.
    Suppose $L^{\mathrm{M}}\nvdash A^\corsiTr$.
    By Theorem \ref{thm:N_FFP}, there is a finite modal FMT frame $F^\mathrm{M}$ validating every formula in $X^\ast$ in which $A^\corsiTr$ is not valid.
    Thus there are a modal FMT model $M^\mathrm{M}$ based on $F^\mathrm{M}$ and $x \in W$ such that $M^\mathrm{M},x \nforces A^\corsiTr$.

    We construct $M^\mathrm{P}$ from $M^\mathrm{M}$ as in Lemma \ref{lem:modal_to_prop_FMT}.
    By the lemma, $M^\mathrm{P},x \nforces A$.

    We show that the frame $F^\mathrm{P}$ of $M^{\mathrm{P}}$ is an $L^{\mathrm{P}}$-frame.
    Let $\neg C\in X$.
    Since $\neg C^\corsiTr\in X^\ast$, we have $M^\mathrm{M}\models\neg C^\corsiTr$.
    By Lemma \ref{lem:modal_to_prop_closed_negative}, $M^\mathrm{P}, y \nVdash C$ for all $y \in W$.
    Thus, $F^\mathrm{P} \models \neg C$ by Proposition \ref{prop:closed_negative} (1).
    Therefore it is an $L^{\mathrm{P}}$-frame.
    By the soundness of $L^{\mathrm{P}}$, we obtain $L^{\mathrm{P}}\nvdash A$.

    \medskip

    $(2\Rightarrow 3)$: Obvious.

    \medskip

    $(3\Rightarrow 1)$: We prove the contraposition.
    Suppose $L^{\mathrm{P}}\nvdash A$.
    By Theorem \ref{thm:finite_frames}, there are a finite $L^{\mathrm{P}}$-frame $F^{\mathrm{P}}$, a propositional FMT model $M^\mathrm{P}$ based on $F^{\mathrm{P}}$, and $x \in M^{\mathrm{P}}$ such that $M^\mathrm{P},x\nforces A$.

    We construct $M^\mathrm{M}$ from $M^\mathrm{P}$ as in Lemma \ref{lem:prop_to_modal_FMT}.
    By the lemma, $M^\mathrm{M},x \nforces A^\corsiTr$.

    We show that $M^\mathrm{M}$ validates all elements of $X^\ast$.
    Let $\neg C\in X$.
    Since $F^\mathrm{P}$ is an $L^{\mathrm{P}}$-frame, $F^\mathrm{P}\models\neg C$.  By Proposition
    \ref{prop:closed_negative} (1), $M^\mathrm{P}, y \nVdash C$ for all $y \in W$.
    Hence, by Lemma \ref{lem:prop_to_modal_FMT}, $M^\mathrm{M}, y \nVdash C^\corsiTr$ for all $y \in W$.
    This means $M^\mathrm{M}\models\neg C^\corsiTr$.

    We fix $k \geq 0$ and we show that $F^\mathrm{M}$ validates $\neg \Box \neg^{2k} \bot$.
    It is easily shown that $\neg^{2k} \bot$ is not of the form $C^\corsiTr \to D^\corsiTr$.
    By the construction in Lemma \ref{lem:prop_to_modal_FMT}, we have $y R^\mathrm{M}_{\neg^{2k} \bot} y$ and $M^{\mathrm{M}}, y \nVdash \neg^{2k} \bot$ for all $y \in W$.
    Then, $F^\mathrm{M}\models\neg\Box \neg^{2k} \bot$.
    We conclude $L^{\mathrm{M}}+ \{\neg\Box \neg^{2k} \bot \mid k \geq 0\} \nvdash A^\corsiTr$.
\end{proof}

\begin{cor}
    For any $A\in\PFml$, the following are equivalent:
    \begin{enumerate}
        \item $\LogicVF\vdash A$.
        \item $\Logic{N}\vdash A^\corsiTr$.
        \item $\Logic{N}+\{\neg\Box \neg^{2k} \bot \mid k \geq 0\} \vdash A^\corsiTr$.
    \end{enumerate}
\end{cor}

\begin{cor}
    For any $A\in\PFml$, the following are equivalent:
    \begin{enumerate}
        \item $\LogicVF+\neg\neg\top\vdash A$.
        \item $\Logic{N}+\neg\Box\neg\Box\top \vdash A^\corsiTr$.
        \item $\Logic{N}+\neg\Box\neg\Box\top+\{\neg\Box \neg^{2k} \bot \mid k \geq 0\} \vdash A^\corsiTr$.
    \end{enumerate}
\end{cor}

The first corollary is obtained by taking $X=\emptyset$.
Hence each logic $L$ with $\Logic{N} \subseteq L \subseteq \Logic{N}+\{\neg\Box \neg^{2k} \bot \mid k \geq 0\}$ is a modal companion of $\LogicVF$.
For example, $\Logic{N}+\neg\Box\bot$, $\Logic{N}+\neg\Box \neg \top$, and $\Logic{N}+\neg\Box\bot + \neg \Box \neg \top$ are modal companions of $\LogicVF$.
The second corollary is obtained by taking $X=\{\neg\neg\top\}$.

Although $\bot$, $\neg\top$, and $\neg\Box\top$ are equivalent over $\Logic{N}$, the formulas $\neg\Box\bot$, $\neg\Box\neg\top$, and $\neg\Box\neg\Box\top$ are not equivalent over $\Logic{N}$.
This is because the modal FMT semantics deals with the relations indexed by $\bot$, $\neg\top$, and $\neg\Box\top$ as independent relations.
The preceding corollaries show that, with respect to Corsi's modified G{\"o}del translation, adding $\neg\Box\bot$ and $\neg\Box\neg\top$ to $\Logic{N}$ does not change the corresponding subintuitionistic logic.
By contrast, $\neg\Box\neg\Box\top$ corresponds to adding the closed negative axiom $\neg\neg\top$ to $\LogicVF$.

\section{Concluding remarks and Future work} \label{sect:conclude}

In this paper, we have introduced a very weak propositional logic $\LogicVF$ and its extensions.
We have also discussed fundamental facts concerning its semantics and its connections to modal logics, especially the pure logic of necessitation $\Logic{N}$.
The study of these logics is still in the early stages, and we believe it can be developed in many directions.
We conclude the present paper by mentioning the following three directions for future work.

The first direction is to develop a sequent calculus for $\LogicVF$.
In particular, it would be natural to find a sequent calculus adapted to FMT semantics in the style of labelled sequent calculi (cf.~\cite{Neg05, DN12}).
This is one of our ongoing projects.


The second is to investigate what happens when adding axioms other than seriality, or more generally axioms other than closed negative axioms.
Candidate axioms are those introduced by Corsi \cite{Cor87} for reflexivity, transitivity, and other properties of Kripke semantics.
From the perspective of modal logic, extensions of the pure logic of necessitation by transitivity axioms are discussed in \cite{Kur24,KS26}.

The third, which arises as a natural question, concerns whether it is possible to introduce a logic weaker than $\LogicVF$ and justify the existence of such a logic by some semantics or other means.

\section*{Acknowledgments}

The first author was supported by JSPS KAKENHI Grant Number JP23K03200.

\bibliographystyle{plain}
\bibliography{refs}

\appendix

\section{Finite frame property for closed modal extensions of $\Logic{N}$}
\label{appendix:N_closed_extensions}

Kogure and Kurahashi \cite{KK25} proved the finite frame property for the logic $\Logic{N}+\neg\Box\bot$.
In this appendix, we prove Theorem \ref{thm:N_FFP}, which extends their theorem by dealing with finite sets of closed modal axioms.

Let $Y$ be a finite set of closed modal formulas and let $L=\Logic{N}+Y$.
For a modal formula $B$, we define ${\sim}B$ as follows: if $B$ is of the form $\neg C$, then ${\sim}B$ is $C$; otherwise, ${\sim}B$ is $\neg B$.
For a modal formula $A$, let $\Subfml_Y(A)$ be the smallest set of modal formulas satisfying the following conditions:
\begin{itemize}
    \item $A\in\Subfml_Y(A)$ and $B\in\Subfml_Y(A)$ for every $B\in Y$,
    \item $\Subfml_Y(A)$ is closed under taking subformulas,
    \item if $B\in\Subfml_Y(A)$, then ${\sim}B \in\Subfml_Y(A)$.
\end{itemize}
Notice that $\Subfml_Y(A)$ is finite because $Y$ is finite.

A subset $\Gamma$ of $\Subfml_Y(A)$ is said to be \emph{$L$-consistent} if $L\nvdash \lconj\Gamma\limp\bot$.
We say that $\Gamma$ is \emph{$(A,Y)$-maximal $L$-consistent} if $\Gamma\subseteq\Subfml_Y(A)$, $\Gamma$ is $L$-consistent, and for every $B\in\Subfml_Y(A)$, either $B\in\Gamma$ or ${\sim}B \in\Gamma$.
It is easy to show that every $L$-consistent subset of $\Subfml_Y(A)$ is included in an $(A,Y)$-maximal $L$-consistent set.

\begin{thm}\label{thm:N_FFP_appendix}
    Let $Y$ be a finite set of closed modal formulas and let $L=\Logic{N}+Y$.
    For any modal formula $A$, the following are equivalent:
    \begin{enumerate}
        \item $L\vdash A$.
        \item $F\models A$ for all modal FMT frames $F$ validating all formulas
              in $Y$.
        \item $F\models A$ for all finite modal FMT frames $F$ validating all
              formulas in $Y$.
    \end{enumerate}
\end{thm}

\begin{proof}
    The implication $(1\Rightarrow 2)$ follows from the soundness of $\Logic{N}$
    with respect to modal FMT semantics.
    The implication $(2\Rightarrow 3)$ is immediate.

    \medskip

    $(3\Rightarrow 1)$: We prove the contraposition of this direction.
    Suppose $L\nvdash A$.
    Then $\{{\sim}A\}$ is $L$-consistent.
    We find an $(A,Y)$-maximal $L$-consistent set $\Gamma_A$ such that ${\sim}A\in\Gamma_A$.

    We define a modal FMT model
    $M_A^Y=\langle W,\{R_B\}_{B\in\MFml},V\rangle$ as follows.
    \begin{itemize}
        \item $W$ is the set of all $(A,Y)$-maximal $L$-consistent sets.
        \item For $\Gamma,\Delta\in W$, define
              $\Gamma R_B\Delta : \iff \Box B\notin\Gamma$ or $B\in\Delta$.
        \item $\Gamma\in V(p) : \iff p\in\Gamma$.
    \end{itemize}
    Since $\Subfml_Y(A)$ is finite, $W$ is finite.

    We prove the following truth lemma by induction on the construction of $B$: for every $B\in\Subfml_Y(A)$ and every
    $\Gamma\in W$,
    \[
        M_A^Y,\Gamma\forces B\ \text{if and only if}\ B\in\Gamma.
    \]
    The cases of propositional variables, $\bot$, and Boolean connectives are standard from maximal $L$-consistency.
    We only prove the case $B\equiv\Box C$.

    \medskip

    $(\Rightarrow)$: We prove the contraposition.
    Suppose $\Box C\notin\Gamma$.
    Since $\Gamma$ is $(A,Y)$-maximal, $\sim\Box C\in\Gamma$.
    We show that $\{\sim C\}$ is $L$-consistent.
    If not, then $L\vdash C$, and hence $L\vdash\Box C$ by necessitation.
    This contradicts the $L$-consistency of $\Gamma$, because $\sim\Box C\in\Gamma$.  Thus $\{{\sim}C\}$ is $L$-consistent.
    Then, we find $\Delta\in W$ such that ${\sim}C\in\Delta$.
    Since $C\notin\Delta$, by the induction hypothesis, we obtain $M_A^Y,\Delta\nforces C$.
    Also, $\Gamma R_C\Delta$ holds because $\Box C\notin\Gamma$.
    Therefore $M_A^Y,\Gamma\nforces\Box C$.

    \medskip

    $(\Leftarrow)$:
    Suppose $\Box C\in\Gamma$.
    Let $\Delta\in W$ be such that $\Gamma R_C\Delta$.
    By the definition of $R_C$, we have $C\in\Delta$.
    By the induction hypothesis, $M_A^Y,\Delta\forces C$.
    Therefore $M_A^Y,\Gamma\forces\Box C$.

    \medskip

    Since ${\sim}A\in\Gamma_A$, we have $A\notin\Gamma_A$ by the $L$-consistency of $\Gamma_A$.
    Hence, by the truth lemma, $M_A^Y,\Gamma_A\nforces A$.

    We show that the frame of $M_A^Y$ validates every formula in $Y$.
    Let $B\in Y$.
    Since $L\vdash B$, every $(A,Y)$-maximal $L$-consistent set $\Gamma$ contains $B$.
    By the truth lemma, $B$ is valid in $M_A^Y$.
    Since $B$ is closed, $B$ is valid in the frame of $M_A^Y$.
\end{proof}
\end{document}